\documentclass[12pt]{article}
\usepackage{amsmath,amssymb,amsthm}
\usepackage{url}
\newtheorem{lem}{Lemma}

\newtheorem*{thm*}{Theorem}
\newtheorem*{prop*}{Proposition}
\newtheorem*{cor*}{Corollary}
\theoremstyle{definition}

\theoremstyle{definition}
\newtheorem{rmk}{Remark}
\theoremstyle{definition}

\begin{document}

\author{Hanno von Bodecker\footnote{Fakult{\"a}t f{\"u}r Mathematik, Ruhr-Universit{\"a}t Bochum, 44780 Bochum, Germany}}

\title{On the $f$-invariant of products}

\date{}

\maketitle

\begin{abstract}
The $f$-invariant is a higher version of the $e$-invariant that takes values in the divided congruences between modular forms; in the situation of a cartesian product of two framed manifolds, the $f$-invariant can actually be computed from the $e$-invariants of the factors. The purpose of this note is to determine the $f$-invariant of all such products.
\end{abstract}

\section{Introduction and statement of the result}


In order to gain a better understanding of the stable homotopy groups of the sphere, which, by the Pontrjagin-Thom construction, can be interpreted as the bordism groups of framed manifolds, it proves helpful to organize the information using suitable invariants. In his seminal work on the $J$-homomorphism, Adams considered the $d$-invariant, which is essentially the degree of a map, and introduced the notion of a so-called $e$-invariant \cite{Adams:1966ys}. Using complex $K$-theory, he produced an invariant 
\begin{equation}
e_{\mathbb{C}}:\pi^{st}_{2k+1}\rightarrow \mathbb{Q/Z},
\end{equation}
and determined its image; to be precise, $e_{\mathbb{C}}$ maps:
\begin{itemize}
\item[(i)] $\pi^{st}_{8k+1}$ onto integer multiples of $\frac{1}{2}$,
\item[(ii)] $\pi^{st}_{8k+3}$  onto integer multiples of $\frac{B_{4k+2}}{4k+2}$,
\item[(iii)] $\pi^{st}_{8k+5}$ onto integer multiples of $1$,
\item[(iv)] $\pi^{st}_{8k+7}$ onto integer multiples of $\frac{B_{4k+4}}{8k+8}$.
\end{itemize}
Moreover, he showed that the optimal values in (ii) and (iv) are attained on the generators of Im$J$  in the corresponding dimensions (although Im$J_{8k+3}$ actually has {\em twice} the order of the denominator of $B_{4k+2}/(4k+2)$, which can be seen by making use of $KO$-theory), and constructed an 8-periodic family $\mu_{8k+1}$  ($\in$ coker$J_{8k+1}$ if $k>1$) detected by (i). The subsequent work of Conner and Floyd using bordism theory allowed the $e_{\mathbb{C}}$-invariant to be computed via the Todd genus of a complex null-bordism of the framed manifold in question \cite{Conner:1966jw}; even later, the work of Atiyah, Patodi, and Singer on index theory on manifolds with boundary \cite{Atiyah:1975kl} led to an analytic formulation of the $e$-invariant \cite{Atiyah:1975ai}, which was used by Deninger and Singhof to exhibit an infinite family of nilmanifolds representing (twice) the generator of Im$J$ in dimensions $8k+3$ ($8k+7$) \cite{Deninger:1984zt}.


Much of the recent progress on our knowledge of $\pi^{st}$ can be accredited to the Adams--Novikov spectral sequence (ANSS), 
$$E_{2}^{p,q}=Ext^{p,q}_{MU_*MU}\left(MU_*,MU_*\right)\Rightarrow\pi_{q-p}^{st}$$
and the rich algebraic structure inherent in complex oriented cohomolgy theories (see e.g.\ \cite{Ravenel:2004xh}). Within this framework, the complex $e$-invariant can be considered as taking values in the $1$-line of the ANSS, i.e.\
$$e_{\mathbb{C}}:\pi^{st}_{2k+1}\rightarrow E_{2}^{1,2k+2}\subseteq \mathbb{Q/Z}.$$
Working locally at a prime $p$ and switching to $BP$, the $1$-line is generated by the so-called alpha elements $\alpha_{i/j}\in Ext^{1,2\left(p-1\right)i}$ of order $p^j$; at odd primes, $j=1+\nu_p\left(i\right)$, and all these alphas are permanent and detect the generators of Im$J_{\left(p\right)}$. At $p=2$ however, the situation is a little bit more subtle: The elements $\alpha_{4k+3}$ do not survive; furthermore  $j=2+\nu_2\left(i\right)$ for $i=2t>2$, but for $t=2k+1$ it is actually $\alpha_{4k+2/2}$ (i.e.\ twice the generator of $Ext^{1,8k+4}$ if $k>0$) that is a permanent cycle represented by an element in Im$J$ of order eight.


In order to detect second filtration phenomena, Laures introduced the $f$-invariant, which 
is a follow-up to the $e$-invariant and takes values in the divided congruences between modular forms \cite{Laures:1999sh,Laures:2000bs}. Let us briefly recall its definition: Considering the congruence subgroup $\Gamma=\Gamma_1(N)$, set $\mathbb{Z}^{\Gamma}=\mathbb{Z}[\zeta_N,1/N]$ and denote by $M^{\Gamma}_*$ the graded ring of modular forms w.r.t.~$\Gamma$ which expand integrally, i.e.~which lie in $\mathbb{Z}^{\Gamma}[\![q]\!]$. The ring of {\em divided congruences} $D^{\Gamma}$  consists of those rational combinations of modular forms which expand integrally; this ring can be filtered by setting
$$D_k^{\Gamma}=\left\{\left.f={\textstyle{\sum_{i=0}^{k}}}f_i\ \right| f_i\in M_i^{\Gamma}\otimes\mathbb{Q},\ f\in\mathbb{Z}^{\Gamma}[\![q]\!]\right\}.$$
Finally, introduce $$\underline{\underline{D}}^{\Gamma}_{k}=D^{\Gamma}_k+M_0^{\Gamma}\otimes\mathbb{Q}+M_k^{\Gamma}\otimes\mathbb{Q}.$$

Now, if $Ell^{\Gamma}$ denotes the complex oriented elliptic cohomology theory associated to the universal curve over the ring of modular forms w.r.t.\ $\Gamma$, the composite
$$E_2^{2,2k+2}[MU]\rightarrow E_2^{2,2k+2}[Ell^{\Gamma}]\rightarrow\underline{\underline{D}}^{\Gamma}_{k+1}\otimes{\mathbb{Q/Z}}$$
is injective (away from primes dividing the level $N$) and can be combined with the composite $\pi^{st}_{2k}\rightarrow E^{2,2k+2}_{\infty}[MU]\rightarrow E^{2,2k+2}_2[MU]$. Then, for a fixed level $N$ (which we suppress from the notation), the $f$-invariant becomes a map
\begin{equation}
f:\pi^{st}_{2k}\rightarrow\underline{\underline{D}}^{\Gamma}_{k+1}\otimes{\mathbb{Q/Z}}.
\end{equation}


 In a more geometrical fashion, the $f$-invariant can be formulated as an elliptic genus of manifolds with corners of codimension two \cite{Laures:2000bs}; as shown in the author's thesis \cite{Bodecker:2008pi}, this allows to use techniques from index theory to understand and, at least in some cases, to calculate the $f$-invariant analytically (see also \cite{Bunke:2008} for a -- slightly different -- analytical approach to the $f$-invariant).
In the situation of the product of two framed manifolds, the $f$-invariant can actually be computed from the $e_{\mathbb{C}}$-invariants of the factors. The purpose of this note is to determine the $f$-invariant of all such products:


\begin{thm*}
Let $x_{4k-1}$ be a generator of {\em Im$J$} in dimension $4k-1$ and let $\mu_{8k+1}$ be a representative of the $\mu$-family constructed by Adams. Then, for the level $N=3$, we have:
\begin{itemize}
\item[\textup{(i)}] $f(x_3^2)\equiv\frac{1}{2}\left(\frac{E_1^2-1}{12}\right)^2$,
\item[\textup{(ii)}] $f(x_7^2)\equiv\frac{1}{2}\left(\frac{E_4-1}{240}\right)^2$,
\item[\textup{(iii)}] $f(\mu_{8k+1}\mu_{8k'+1})\equiv\frac{1}{2}\frac{E_1-1}{2}$,
\item[\textup{(iv)}] $f(\mu_{8k+1}x_{8k'-1})\equiv\frac{1}{2}\frac{E_4-1}{240}$,
\item[\textup{(v)}] $f(\mu_{8k+1}x_{8k'+3})\equiv0$.
\end{itemize}
Furthermore, for {\em{any}} level $N>1$, we have:
\begin{itemize}
\item[\textup{(vi)}]$f(x_{4k-1}x_{4k'-1})\equiv0,$
\end{itemize}
unless $k=k'$ equals one or two.
\end{thm*}


Some comments are in order before presenting the proof: 

First of all, we have to admit that at {\em even} levels $N=2l$, i.e.\ when two is inverted, our theorem  does not provide anything new, since the products of alpha elements at odd primes are known to vanish \cite{Miller:1977ya} (although our result does not rely on this fact). At odd levels however, things become more interesting:

It might be somewhat amusing to find the two $8$-periodic families (iii) and (iv), each (within its family) admitting the same representative of the $f$-invariant, but the experienced reader will surely recognize these as corresponding to $\mu_{8\left(k+k'\right)+2}\in$ coker$J$ and Im$J_{8\left(k+k'\right)}$, respectively. 

Furthermore, our result shows that beyond these two families, only the exceptional cases (i) and (ii) (corresponding to the Kervaire elements of product type) admit a non-trivial $f$-invariant.

Finally, we would like to point out that the choice $N=3$ made in results (i) through (v) is simply a matter of convenience; similar statements can be established at all odd levels.

Summarizing, and bearing in mind that we work on the level of manifolds, i.e.\ representatives of permanent cycles, this theorem may be thought of as an elliptic  (and strengthened) analog of \cite[Theorem 5.5.8.]{Ravenel:2004xh}.

\section{Proof of the Theorem}


The determination of the $f$-invariant of a product from the $e_{\mathbb{C}}$-invariants of its factors is made possible by the following result:

\begin{lem}\textup{\cite{Bodecker:2008pi}}\label{f_from_e}
Let $Y_1$, $Y_2$ be odd-dimensional framed manifolds, and let $m(Y_i)$ be {\em any} modular form of weight  $(\dim Y_i +1)/2$ w.r.t.~the fixed congruence subgroup $\Gamma=\Gamma_1(N)$ such that $\bar{m}(Y_i)=m(Y_i)-e_{\mathbb{C}}(Y_i)\in\mathbb{Z}^{\Gamma}[\![q]\!]$. Then we have
$$f(Y_1\times Y_2)\equiv \bar{m}(Y_1) e_{\mathbb{C}}(Y_2)\equiv-\bar{m}(Y_2) e_{\mathbb{C}}(Y_1).$$
In particular, the $f$-invariant of a product is antisymmetric under exchange of the factors.
\end{lem}


Besides the knowledge of the possible values of the $e_{\mathbb{C}}$-invariant, we need some elementary number theory:
 
\begin{lem}\label{euler-fermat}
For $n\geq1$, we have:
\begin{itemize}
\item[\textup{(i)}] $\left(d^{p^{n-1}\left(p-1\right)}-1\right)d^{n}\equiv0\mod p^{n}$,
\item[\textup{(ii)}] $\left(d^{\left(2n'+1\right)2^n}-1\right)d^{n+2}\equiv0\mod2^{n+2}$.
\end{itemize}
\end{lem}
\begin{proof}
Part (i) is a simple consequence of the Euler--Fermat theorem \cite{Apostol:1976qf} which states that for $(a,m)=1$, we have:
$$a^{\varphi(m)}\equiv1\mod m,$$
where  $\varphi(m)=m\prod_{p|m}\left(1-\frac{1}{p}\right)$ is the Euler totient. 

The refinement (ii) stems from the fact that for odd $d^{2n'+1}=2k+1$:
\begin{equation*}
\begin{split}
(2k+1)^{2^n}&=1+2^{n}2k+\frac{2^n\left(2^n-1\right)}{2}4k^2+\sum_{i\geq3}\left(\begin{array}{c} 2^n\\ i\end{array}\right)\left(2k\right)^i\\
&=1+2^{n+1}\left(k+\left(2^n-1\right)k^2\right)+\sum_{i\geq3}\left(\begin{array}{c} 2^n\\ i\end{array}\right)\left(2k\right)^i\\
&\equiv1\mod2^{n+2},\\
\end{split}
\end{equation*}
since $\nu_2\left(i!\right)<i$ and $\nu_2\left(\left(2^n\right)!/\left(2^n-i\right)!\right)\geq n+1$ for $i\geq3$.
\end{proof}


Concerning modular forms, recall that for even $k>2$ the Eisenstein series
$$E_k=-\frac{2k}{B_k}G_k=1-\frac{2k}{B_k}\sum_{n=1}^{\infty}{\sigma_{k-1}(n)q^n}$$
is a modular form of weight $k$ w.r.t.\ the {\em full} modular group. Furthermore recall that the ring of modular forms w.r.t.~$\Gamma=\Gamma_1(3)$ is generated by
\begin{equation*}
\begin{split}
E_1&=1+6\sum_{n=1}^{\infty}\sum_{d|n}(\frac{d}{3})q^n,\\
E_3&=1-9\sum_{n=1}^{\infty}\sum_{d|n}(\frac{d}{3})d^2q^n,\\
\end{split}
\end{equation*}
where $(\frac{d}{3})$ is the Legendre Symbol; these modular forms satisfy the following useful congruence:
\begin{lem}\label{E_1squaredE_3}
$$\frac{1}{2}\left\{\frac{E_1^2-1}{12}+(2k+1)\frac{E_3-1}{9}\right\}\in\mathbb{Z}[\![q]\!].$$
\end{lem}
\begin{proof}
Since $$E_1^2=1+12\sum_{n=1}^{\infty}\sum_{3\nmid d|n}d\ q^n,$$
we have to show that
$$\frac{1}{2}\left\{\sum_{n=1}^{\infty}\sum_{3\nmid d|n}d\ q^n-(2k+1)\sum_{n=1}^{\infty}\sum_{d|n}(\frac{d}{3})d^2q^n\right\}\in\mathbb{Z}[\![q]\!],$$
but obviously
\begin{equation*}
\sum_{d|n}(\frac{d}{3})d^2\equiv\sum_{3\nmid d|n}d^2\equiv\sum_{3\nmid d|n}d \mod2.
\end{equation*}
\end{proof}


Now let us prove the Theorem:

\subsection*{The exceptional cases (i) and (ii)}


Looking at the list, we may assume that $e_{\mathbb{C}}(x_3)$ is given by $-\frac{1}{12}$; of course, this is in accordance with the well-known fact that $x_3$ may be represented by the sphere $S^3$. So  we apply Lemma \ref{f_from_e},  choose $$-\bar{m}(x_3)=\frac{E_1^2-1}{12}=\sum_{n=1}^{\infty}\sum_{3\nmid d|n}d\ q^n\in\mathbb{Z}[\![q]\!],$$
and compute
$$\frac{1}{12}\frac{E_1^2-1}{12}\equiv-\frac{1}{2}\left(\frac{E_1^2-1}{12}\right)^2\equiv\frac{1}{2}\left(\frac{E_1^2-1}{12}\right)^2\mod  \underline{\underline{D}}_4^{\Gamma};$$
this proves (i).

Similarly, $x_7$ may be represented by the sphere $S^7$ (see e.g.\ \cite{Conner:1966jw}), resulting in $e_{\mathbb{C}}(x_7)=\frac{1}{240}$, so  we choose
$$\bar{m}(x_7)=\frac{E_4-1}{240}=\sum_{n=1}^{\infty}\sum_{d|n}d^3q^n\in\mathbb{Z}[\![q]\!]$$ 
and `complete the square', hence establishing (ii).

\begin{rmk} Of course, (i) and (ii) correspond to the Kervaire elements of product type, i.e.\ $\alpha_{2/2}^2=\beta_{2/2}$ and $\alpha_{4/4}^2=\beta_{4/4}$ at the prime $p=2$. In order to see this on the level of $f$-invariants, we may compare (i) and (ii) to the results of \cite{Hornbostel:2007ss}, i.e.\
$$f\left(\beta_{2/2}\right)\equiv\frac{1}{2}\left(\frac{E_1^2-1}{4}\right)^2\in\underline{\underline{D}}_4^{\Gamma}\otimes{\mathbb{Q/Z}},$$
$$f\left(\beta_{4/4}\right)\equiv\frac{1}{2}\left(\frac{E_1^2-1}{4}\right)^4+\frac{1}{2}\left(\frac{E_1^2-1}{4}\right)^3\in\underline{\underline{D}}_8^{\Gamma}\otimes{\mathbb{Q/Z}};$$
but a short computation modulo $\underline{\underline{D}}_8^{\Gamma}$ reveals:
\begin{equation*}
\begin{split}
&\frac{1}{2}\left(\frac{E_4-1}{16}\right)^2\\
=&\frac{1}{2}\left(\frac{1}{2}\frac{E_1^4-1}{8}+\frac{1}{2}\left(E_1^4-E_1E_3\right)\right)^2\\
=&\frac{1}{2}\left(\frac{1}{4}\left(\frac{E_1^4-1}{8}\right)^2+\frac{1}{2}\frac{E_1^4-1}{8}\left(E_1^4-E_1E_3\right)+\frac{1}{4}\left(E_1^4-E_1E_3\right)^2\right)\\
\equiv&\frac{1}{32}E_1E_3+\frac{1}{2}\left(\frac{E_1^2-1}{4}\right)^4+\frac{1}{2}\left(\frac{E_1^2-1}{4}\right)^3+\frac{1}{8}\left(\frac{E_1^2-1}{4}\right)^2\\
\equiv&\frac{1}{16}\frac{E_1^2-1}{4}+\frac{1}{2}\left(\frac{E_1^2-1}{4}\right)^4+\frac{1}{2}\left(\frac{E_1^2-1}{4}\right)^3+\frac{1}{16}\frac{E_1^4-1}{8}-\frac{1}{16}\frac{E_1^2-1}{4}\\
\equiv&\frac{1}{2}\left(\frac{E_1^2-1}{4}\right)^4+\frac{1}{2}\left(\frac{E_1^2-1}{4}\right)^3,\\
\end{split}
\end{equation*}
since
\begin{equation*}
\begin{split}
-\frac{1}{32}E_1E_3&\equiv\frac{1}{2}\frac{E_4-1}{16}E_1E_3\equiv\frac{1}{2}\frac{E_4-1}{16}E_3\\
&\equiv\frac{E_4-1}{16}\frac{1}{2}\left(E_3-1\right)\equiv\frac{1}{2}\frac{E_1^2-1}{4}\frac{E_4-1}{16}\\
&\equiv\frac{1}{2}\frac{E_6-1}{8}\frac{E_1^2-1}{4}\equiv-\frac{1}{16}\frac{E_1^2-1}{4}.\\
\end{split}
\end{equation*}
\end{rmk}

\begin{rmk}
From the point of view of arithmetics, the product-type Kervaire elements ($x_3^2=\nu^2$ and $x_7^2=\sigma^2$) are lucky exceptions: They happen to occur in the dimensions just low enough to be left unscathed by the congruence (ii) of Lemma \ref{euler-fermat}.
\end{rmk}


\subsection*{The 8-periodic families (iii) and (iv)} 
According to \cite{Adams:1966ys}, we know that $e_{\mathbb{C}}(\mu_{8k+1})=1/2$, and this is the only possible non-trivial value of $e_{\mathbb{C}}$ in dimension $8k+1$. Thus, we judiciously choose 
$$f(\mu_{8k+1}\mu_{8k'+1})\equiv\frac{1}{2}\frac{E_1E_4^k-1}{2}\in\underline{\underline{D}}_{4\left(k+k'\right)+2}^{\Gamma}\otimes{\mathbb{Q/Z}},$$
but
$$\textstyle{\frac{1}{2}\frac{E_1E_4^k-1}{2}-\frac{1}{2}\frac{E_1-1}{2}=\frac{E_4^{k}-1}{4}E_1\in\mathbb{Z}[\![q]\!]}$$
yields (iii). For (iv), we may assume that $e_{\mathbb{C}}(x_{8k'-1})\equiv B_{4k'}/8k'\mod\mathbb{Z}$, hence
$$f(\mu_{8k+1}x_{8k'-1})\equiv{\textstyle{\frac{1}{2}\frac{B_{4k'}}{8k'}}}\left(E_{4k'}-1\right)\equiv{\textstyle{\frac{1}{2}}}\sum\sigma_{4k'-1}(n)q^n\in\underline{\underline{D}}_{4\left(k+k'\right)+1}^{\Gamma}\otimes{\mathbb{Q/Z}},$$
which is congruent to the desired result by means of Lemma \ref{euler-fermat}.

\begin{rmk}
In order to express (iv) in terms of $E_1$ and $E_3$, it is useful to compute the following relation modulo $\underline{\underline{D}}_{4k+1}^{\Gamma}$:
\begin{equation*}
\begin{split}
&\frac{1}{2}\frac{E_4-1}{16}\\
=&\frac{1}{2}\left[\left(\frac{E_1^2-1}{4}\right)^2+\frac{1}{2}\frac{E_1^2-1}{4}+\frac{1}{2}\left(E_1^4-E_1E_3\right)\right]\\
=&\frac{1}{2}\left(\frac{E_1^2-1}{4}\right)^2+\frac{1}{2}\left[\frac{1}{2}\frac{E_1^2-1}{4}-\frac{1}{2}\left(E_3-1\right)\right]-\frac{1}{2}\frac{E_1-1}{2}E_3+\frac{E_1^4-1}{4}\\
\equiv&\frac{1}{2}\left(\frac{E_1^2-1}{4}\right)^2+\frac{1}{2}\left[\frac{1}{2}\frac{E_1^2-1}{4}-\frac{1}{2}\left(E_3-1\right)\right]+\frac{1}{2}\frac{E_1-1}{2}\frac{E_1^2-1}{4};\\
\end{split}
\end{equation*}
furthermore, we observe that for $k\geq2$ the first summand can be dropped due to $\frac{1}{2}\left(\frac{E_1^2-1}{4}\right)^2\equiv\frac{1}{2}\left(E_3-1\right)^2\equiv\frac{1}{2}E_3^2\equiv\frac{1}{2}E_1^{4k-5}E_3^2$.
\end{rmk}

\begin{rmk}
The elements occurring in part (iii) and (iv) of the Theorem also allow detection via the invariants $d_{\mathbb{R}}$ and $e_{\mathbb{R}}$, respectively, cf.~\cite{Adams:1966ys}. While for low values of $k$, $k'$, the non-triviality of $f$ in these cases is easily checked by hand, the author does not know how to establish non-triviality for {\em all} values of $k$, $k'$.
\end{rmk}


\subsection*{The generic situation}
 For (v),  we may assume that $e_{\mathbb{C}}(x_{8k+3})$ is represented by $B_{4k+2}/(4k+2)$, so, for $k>0$, we have 
$$f(x_{8k+3}\mu_{8k'+1})\equiv{\frac{1}{2}\frac{B_{4k+2}(1-E_{4k+2})}{4k+2}}=\sum\sigma_{4k+1}(n)q^n,$$
whereas for $k=0$ we have
\begin{equation*}
\begin{split}
f(x_3\mu_{8k'+1})&\equiv{{\frac{1}{2}}}\frac{E_1^2-1}{12}\equiv{\frac{1}{2}}\left\{\frac{E_1^2-1}{12}+E_4^{k'}E_3-1\right\}\\
&\equiv{\frac{1}{2}}\left\{\frac{E_1^2-1}{12}+E_3-1\right\}\mod  \underline{\underline{D}}_{4k'+3}^{\Gamma},\\
\end{split}
\end{equation*}
which expands integrally by Lemma \ref{E_1squaredE_3}.

Turning to (vi) and excluding the cases $k=k'=1$ and $k=k'=2$, we assume that $k\leq k'$; thus
$$f(x_{4k-1}x_{4k'-1})\equiv\left(\epsilon(k')\sum\sigma_{2k'-1}(n)q^n\right)\cdot\left(\epsilon(k)\frac{B_{2k}}{4k}\right),$$
where $\epsilon(k)=1$ for $k$ even and two otherwise.

The Theorem of von Staudt--Clausen \cite{Apostol:1976qf} allows the computation of the denominator of the Bernoulli numbers. More precisely, let $j_{2k}$ denote the denominator of $B_{2k}/2k$; if $(p-1)p^{n-1}|2k$, then $p^n|j_{2k}$, 
and this result is sharp. But  by Lemma \ref{euler-fermat} (i) we have:
$$p^{-n}\sum\sigma_{2k'-1+2k}(r)q^r\equiv p^{-n}\sum\sigma_{2k'-1}(r)q^r\  \text{if}\ (p-1)p^{n-1}|2k,$$
and the LHS is $p^{-n}G_{2k+2k'}$ plus a constant, hence vanishes mod $\underline{\underline{D}}_{2\left(k+k'\right)}$.

Similarly, if $k=2l=(2n'+1)2^{m+1}$, $2^{m+4}|2j_{4l}$ (and $2j_{4l}$ is precisely the order of Im$J$ in dimension $8l-1$), and part (ii) of Lemma \ref{euler-fermat} yields:
$$\frac{1}{2^{4+m}}\left(\sum\sigma_{2k'-1+4l}(r)q^r-\sum\sigma_{2k'-1}(r)q^r\right)\equiv0\mod\mathbb{Z}[\![q]\!].$$
This completes the proof.

\bibliography{refbib_edited}
\end{document}